\def\mode{0}	
\if 0\mode
\documentclass[journal, twoside, web]{ieeecolor}
\usepackage{lcsys}
\else
\documentclass[letterpaper, 10 pt, conference]{ieeeconf}
\IEEEoverridecommandlockouts
\fi

\usepackage[english]{babel}
\usepackage{microtype}
\usepackage{setspace}
\usepackage{blindtext}
\usepackage{siunitx}
\usepackage{titlecaps}
\Addlcwords{or with if of and an a to via -off through off for in into the versus vs subject} 
\usepackage{comment}

\usepackage{graphicx}
\usepackage{color}
\graphicspath{{Images/}}

\usepackage{array}
\usepackage{stfloats}
\usepackage{textcomp}

\usepackage{amsmath,amssymb,amsfonts,amsthm}
\usepackage{mathdots}
\usepackage{bbm,dsfont}
\usepackage{relsize}
\usepackage{nicefrac}
\usepackage{bbm}
\usepackage{mathtools}
\usepackage[mathscr]{eucal}
\usepackage{optidef}

\usepackage{cite}

\usepackage{url}
\usepackage{hyperref}
\makeatletter
\let\NAT@parse\undefined
\makeatother
\usepackage[capitalize]{cleveref}
\newcommand\orcidicon[1]{\href{https://orcid.org/#1}{\includegraphics[scale=0.04]{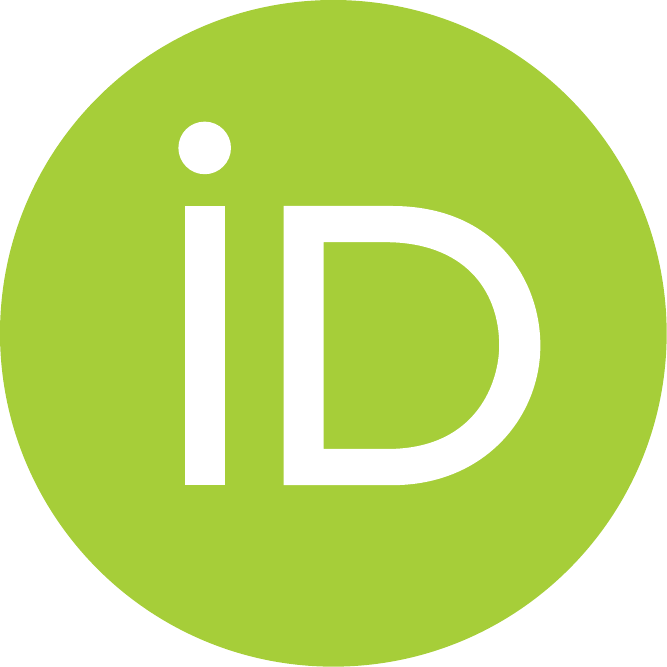}}}

\if0\mode
\newcommand{\myParagraph}[1]{{\color{nblue}\sf\itshape\small\bfseries {#1.}}}
\else
\newcommand{\myParagraph}[1]{{\bfseries {#1.}}}
\fi

\newcommand{\paperType}{\if0\mode letter\xspace \else paper\xspace \fi}

\newcommand{\Real}[1]{ { {\mathbb R}^{#1} } }

\DeclareMathOperator*{\argmin}{arg\;min}

\newcommand{\DeclareAutoPairedDelimiter}[3]{%
	\expandafter\DeclarePairedDelimiter\csname Auto\string#1\endcsname{#2}{#3}%
	\begingroup\edef\x{\endgroup
		\noexpand\DeclareRobustCommand{\noexpand#1}{%
			\expandafter\noexpand\csname Auto\string#1\endcsname*}}%
	\x}

\DeclareAutoPairedDelimiter{\ceil}{\lceil}{\rceil}
\DeclareAutoPairedDelimiter{\floor}{\lfloor}{\rfloor}

\newcommand{\e}{\mathrm{e}}
\newcommand{\one}{\mathds{1}}
\newcommand{\spec}[1]{\sigma(#1)}
\newcommand{\pol}[2]{h_{#1}(z;#2)}

\newcommand{\x}[2]{x_{#1}(#2)}

\newcommand{\net}[1]{\mathcal{G}_{#1}}
\newcommand{\nodes}{\mathcal{V}}
\newcommand{\edges}[1]{\mathcal{E}_{#1}}
\newcommand{\neigh}[2]{\nodes_{#1}(#2)}
\renewcommand{\deg}[2]{d_{#1}(#2)}
\newcommand{\dmax}[1]{D_{#1}}
\renewcommand{\u}[2]{u_{#1}(#2)}

\newcommand{\rate}[1]{R_{#1}}
\newcommand{\up}[1]{\bar{\lambda}_{#1}}
\newcommand{\eig}[2]{\lambda_{#2}^{#1}}
\newcommand{\bareig}[2]{\bar{\lambda}_{#2}^{#1}}

\newcommand{\taun}{\tau_n}

\newcommand{\gpos}{\lambda}

\if 0\mode
\newtheoremstyle{colorthm}{\topsep}{\topsep}{\normalfont}{1em}{\color{nblue}\sf\itshape\small\bfseries}{:}{ }{\thmname{#1}\thmnumber{ #2}{\thmnote{ (#3)}}}
\theoremstyle{colorthm}
\else
\theoremstyle{plain}
\fi

\newtheorem{cor}{Corollary}
\newtheorem{prop}{Proposition}
\newtheorem{lemma}{Lemma}

\newtheorem{prob}{Problem}
\newtheorem{ass}{Assumption}
\newtheorem{rem}{Remark}

\Crefname{cor}{Corollary}{Corollaries}
\Crefname{conj}{Conjecture}{Conjectures}
\Crefname{ass}{Assumption}{Assumptions}
\Crefname{figure}{Fig.}{Figures}

\addto\captionsenglish{}
\addto\captionsenglish{}

\newcommand{\revision}[1]{{\color{black}#1}}

\newcommand{\blue}[1]{{\color{black} #1}}

\newcommand{\linkToPdf}[1]{\href{#1}{\blue{(pdf)}}}
\newcommand{\linkToPpt}[1]{\href{#1}{\blue{(ppt)}}}
\newcommand{\linkToCode}[1]{\href{#1}{\blue{(code)}}}
\newcommand{\linkToWeb}[1]{\href{#1}{\blue{(web)}}}
\newcommand{\linkToVideo}[1]{\href{#1}{\blue{(video)}}}
\newcommand{\linkToMedia}[1]{\href{#1}{\blue{(media)}}}
\newcommand{\award}[1]{\xspace} 
\newcommand{\eg}{\emph{e.g.,}\xspace}
\newcommand{\ie}{\emph{i.e.,}\xspace}
\if0\mode
\addto\extrasenglish{}

\addto\extrasenglish{}
\else
\addto\extrasenglish{}

\fi
\addto\extrasenglish{}
\addto\extrasenglish{}

\if0\mode
\title{\titlecap{Faster consensus via a sparser controller}}
\author{Luca~Ballotta\textsuperscript{\orcidicon{0000-0002-6521-7142}}
	and Vijay~Gupta\textsuperscript{\orcidicon{0000-0001-7060-3956}},~\IEEEmembership{Fellow,~IEEE}
	\thanks{This work was supported in part
		by the Italian Ministry of Education, University and Research (MIUR) through
		the PRIN Project under Grant 2017NS9FEY entitled ``Realtime Control of 5G Wireless Networks'' and through
		the initiative "Departments of Excellence" (Law 232/2016).
	}%
	\thanks{Luca Ballotta is with the Department of Information Engineering, University of Padova, 35131 Padova, Italy
		(e-mail: ballotta@dei.unipd.it).}%
	\thanks{Vijay Gupta is with the Elmore Family School of Electrical and Computer Engineering, Purdue University, West Lafayette, IN, 47907, USA 
		(e-mail: gupta869@purdue.edu).}
	}
\else
\title{\titlecap{Faster consensus via a sparser controller}}
\author{Luca~Ballotta\textonesuperior \xspace and Vijay~Gupta\texttwosuperior
	\thanks{This work was supported in part
		by the Italian Ministry of Education, University and Research (MIUR) through
		the PRIN Project under Grant 2017NS9FEY ``Realtime Control of 5G Wireless Networks'' and
		the initiative "Departments of Excellence" (Law 232/2016).
		Views and opinions expressed in this work are of the authors and may not reflect those of the funding institutions.}%
	\thanks{\textonesuperior Department of Information Engineering, University of Padova, 35131 Padova, Italy.
		\texttt{ballotta@dei.unipd.it}.}%
	\thanks{\texttwosuperior Elmore Family School of Electrical and Computer Engineering, Purdue University, West Lafayette, IN, 47907, USA.
		\texttt{gupta869@purdue.edu}.}
}
\fi

\begin{document}
	
	\bstctlcite{MyBSTcontrol}
	
	\maketitle
	
	\pagestyle{empty}
	\thispagestyle{empty}
	\numberwithin{equation}{section}
	

\begin{abstract}
	In this \paperType,
	we investigate the architecture of an optimal controller
	that \revision{maximizes} the convergence \revision{speed} of a consensus protocol
	with single-integrator dynamics.
	Under the assumption that communication delays increase with the number of hops from which information is allowed to reach each agent,
	we address the optimal control design under delayed feedback
	and	show that the optimal controller features, in general, a sparsely connected architecture.
	
	\if 0\mode
	\fi
	
\end{abstract}

\section{Introduction}\label{sec:intro}

\if0\mode
\IEEEPARstart{C}{onsensus}
\else
Consensus
\fi
of dynamical systems
is a fundamental tool in control theory and applications
that has been extensively studied over the latest few decades~\cite{4118472,XIAO200733}. One implementation issue in distributed control,
arising especially in large-scale systems where communication occurs over wireless channels,
is latency due to data transmission.
Such a latency affects the feedback information exchanged among agents.
When this latency is non-negligible compared to the system dynamics,
a careful design of the controller needs to consider and compensate the feedback delays
to avoid performance degradation.
The presence of delayed feedback information used in control may trigger dynamic modes
that force control actions to be conservative in order to ensure stability, thus,
making it more difficult to command the system trajectory.

Classical control literature 
addresses this problem by assuming distributed controllers with given architecture (\textit{structured controllers})
where the design of feedback gains takes into account communication delays in the system dynamics.
A large body of work addresses stability conditions. 
Relevant to the consensus problem that we concentrate on, 
a few related works are~\cite{MUNZ20101252} that studies stability of consensus under several delay models for a network of identical agents in continuous time,
\cite{ren2017finite, SUN2021419} that deal with finite-time stabilization of discrete-time systems,
and~\cite{bamjovmitpat12, Chehardoli2019} that consider error compensation in vehicular platoons with different network topologies.
Other relevant works focus on performance,
\eg~\cite{nedic2008convergence} finds an upper bound on the convergence rate of consensus under time-varying delays,
\cite{8430769,8485772} minimize the $ \mathcal{H}_2 $-norm associated with the consensus error,
and~\cite{LI2018144} proposes a control protocol for high-order systems to maximize convergence speed.

A second,
more recent, line of work addresses design not only of control gains
but also of the controller architecture.
For example,
work~\cite{8485772} proposes greedy algorithms that modify the communication links to decrease the $H_2$-norm of the consensus error under time-delays.
The authors in~\cite{ANDERSON2019364} introduce the \textit{System Level Synthesis} as
a possible framework that accounts for impact of communication locality in robust control design.
The article~\cite{JOVANOVIC201676} surveys methods that trade
controller complexity 
for closed-loop performance of stochastically forced systems,
where the optimization problem associated with control design incorporates a regularization term that penalizes the presence of communication links.
Work~\cite{8047285} investigates algorithms for near-optimal edge selection to maximize convergence rate of a multi-robot system to a rigid formation.

While the latter body of literature mostly works under the common wisdom that reducing the total number of communication links entails advantages,
this is typically intended as a benefit from scalability or resource allocation standpoint
rather than to \textit{performance}.
In particular,
the all-to-all architecture 
is typically regarded as optimal for closed-loop performance~\cite{8485772,JOVANOVIC201676},
whereas practical constraints impose sparser implementations in applications.
This stems from the hidden assumption that communication delays
do not depend on the controller architecture in relevant manner.
However,
if 
this did not hold true,
the optimal architecture may be crucially different. 
For example,
work~\cite{gupta2010delay} characterized consensus on lattices under time-slotted communication
where feedback delays depend on transmission power,
proving that the controller architecture maximizing convergence speed is sparse when the lattice dimension is greater than one.
Under similar spirit,
recent work~\cite{ballotta2023tcns} considered mean-square consensus on undirected graphs
where feedback delays depend on the communication hops over a given network,
showing that the optimal controller features sparse interconnections if delays increase fast enough with the number of hops.

In this \paperType,
we draw inspiration from the setup in~\cite{ballotta2023tcns}
to investigate 
performance of distributed controllers 
with respect to the convergence rate of a consensus protocol
under architecture-dependent communication latency.
Differently from~\cite{gupta2010delay},
we consider system dynamics that induce more restrictive stability conditions,
assume a more general model for delays,
and address optimization of feedback gains.
Importantly,
rather than purposely modifying the communication links to improve performance,
as done in, \eg~\cite{8485772},
we consider different architectures with no explicit relation to performance
and explore the effect of architecture-dependent delays.
In analogy to the fundamental performance trade-off observed for stochastic systems in~\cite{ballotta2023tcns},
we show that,
\textit{when delays increase with the density of the architecture,
the distributed controller that minimizes the convergence rate has in general a sparse architecture,
irrespectively of the specific network topology.}
In particular,
numerical experiments 
yield fundamentally different conclusions than~\cite{gupta2010delay} where the optimal architecture was proven
to be the complete graph for the ring topology. 

This \paperType is organized as follows.
\autoref{sec:setup} introduces the system setup. 
\autoref{sec:design} tackles optimal feedback gain design for a distributed controller.
\autoref{sec:simulations} presents numerical performance with different architectures,
showing optimality of sparse controllers.
Conclusions are drawn in~\autoref{sec:conclusion}.

\section{Setup}\label{sec:setup}

\subsection{{System Model}}\label{sec:model}

We address a Networked Control System composed of $ N $ interconnected agents (or nodes)
which aim to consensus.

\myParagraph{Agent Dynamics}
Each agent $ i\in\{1,\dots,N\} $ evolves as a scalar discrete-time single integrator,
\begin{equation}\label{eq:dynamics}
	\x{i}{k+1} = \x{i}{k} + \u{i}{k},
\end{equation}
where $ k\ge0 $ denotes time,
$ \x{i}{k} \in\Real{} $ is the state of agent $ i $,
and $ \u{i}{k} \in\Real{} $ is its control input.
We denote by
$ \x{}{k}\in\Real{N} $ and $ \u{}{k}\in\Real{N} $
the states and inputs of all agents,
respectively.

\myParagraph{Feedback Control}
Agents exchange state information according to a communication network modeled
as a graph $ \net{n} $,
where $ n $ \revision{parametrizes the number of links in $\net{n}$.}%
\footnote{
	In the following,
	we interchangeably use the phrases \emph{communication network},
	\emph{network topology},
	and \emph{controller architecture} (or just \emph{architecture}),
	\revision{by which we refer to the graph that
	describes data exchange among agents.}
	A \emph{\revision{distributed controller}} \revision{is defined by} both its \emph{architecture} and \emph{feedback gains}.
}
\begin{ass}[Distributed controller architecture]\label{ass:architecture}
	The controller architecture is given by the undirected graph $ \net{n} \doteq (\nodes,\edges{n}) $
	\revision{where $ \nodes = \{1,\dots,N\} $ and 
	$ \edges{n} \subseteq \nodes \times \nodes $ is a collection of node pairs
	such that nodes $ i $ and $ j $ communicate if and only if $ (i,j) \in \edges{n} $.
	For $n>1$,
	the edge set is built as}
	\begin{equation}\label{eq:edges}
		\edges{n+1} = \edges{n} \cup \bigcup_{i\in\nodes}\bigcup_{j\in\neigh{n}{i}}\bigcup_{\ell\in\neigh{n}{j}\setminus(\neigh{n}{i}\cup\{i\})} (i,\ell),
	\end{equation}
	where \revision{$\net{1}$} is 
	assigned \emph{a priori} and $ \neigh{n}{i} $ is the \emph{$ n $-neighborhood} of node $ i $,
	\ie
	$ \neigh{n}{i} \doteq \{j\in\nodes : (i,j)\in\edges{n}\} $.
	Graph $ \net{1} $ is connected and simple,
	and $ \deg{n}{i} \doteq |\neigh{n}{i}| $.
\end{ass}
\revision{In words,
given an initial architecture $\net{1}$,
$\net{n}$ is built by connecting nodes whose distance in $\net{1}$ is at most $n$ hops.}

Given 
an architecture $ \net{n} $,
agent $ i $ computes
control inputs $ \u{i}{\cdot} $ via state measurements
received from its $ n $-neighbors.
\begin{ass}[Communication delays~{\cite{ballotta2023tcns}}]\label{ass:delays}
	State measurements communicated across $ \net{n} $
	are received after delay $ \taun = f(n) $,
	where $ f(\cdot) $ is a positive increasing sequence.
\end{ass}
In order to let agents achieve consensus,
we assume Laplacian-type proportional feedback control\revision{~\cite{MUNZ20101252,8485772}},
that is,
\begin{equation}\label{eq:control}
	\u{}{k} = -K_n\x{}{k-\taun},
\end{equation}
where the feedback delay $ \taun > 0 $ follows from~\cref{ass:delays}.

\begin{ass}[Structure of feedback gains]\label{ass:laplacian-K}
	Matrix $ K_n $ satisfies $ K_n = K_n^{\top} $
	and $ K_n\one = 0 $,
	$ \one \in \Real{N} $ being the vector of all ones.
	The feedback gain $ [K_n]_{ij} = [K_n]_{ji} $ is nonzero
	only if the communication link $ (i,j) $ belongs to $ \edges{n} $
	or if $ j = i $.
	Also,
	we require $ [K_n]_{ii} = -\sum_{j\in\neigh{n}{i}}[K_n]_{ij} $ for all $ i\in\nodes $.
\end{ass}

Finally,
agent evolution~\eqref{eq:dynamics} paired with controller~\eqref{eq:control}
yields the following global \revision{consensus} dynamics,
\begin{equation}\label{eq:controlled-dynamics}
	\x{}{k+1} = \x{}{k} - K_n\x{}{k-\taun}.
\end{equation}

\subsection{{Problem Formulation}}\label{sec:problem-formulation}

In this \paperType,
we are primarily interested in investigating the \textit{controller architecture}
that yields the fastest convergence of the consensus dynamics~\eqref{eq:controlled-dynamics}.
We also address \textit{optimal feedback gains} that maximize the consensus speed, 
which allows for a fair comparison of controllers with different architectures.

\revision{It is well known that the convergence rate of the delay-free autonomous dynamics corresponding to~\eqref{eq:controlled-dynamics}
is geometric and given by the Second Largest Eigenvalue Modulus (SLEM) of the state matrix.
In the presence of delays,
system~\eqref{eq:controlled-dynamics} can be rewritten as a delay-free system
by means of state augmentation~\revision{\cite{gupta2010delay}},
where the augmented state $ \x{a}{k} $ stacks $ \taun + 1 $ consecutive states from time $ k - \taun $ to time $ k $.
The convergence rate is then the SLEM of the augmented state matrix $A_n$.
Standard computations yield the following result.

\begin{lemma}[Eigenvalues of delay system~\cite{ballotta2023tcns}]\label{prop:eigenvalues}
	Let~\eqref{eq:controlled-dynamics} be equivalently written as the following delay-free system,
	\begin{equation}\label{eq:augmented-system}
		\x{a}{k+1} = A_n\x{a}{k}.
	\end{equation}
	Then, the spectrum of $ A_n $ is given by
	\begin{equation}\label{eq:augmented-system-eigs}
		\spec{A_n} = \bigcup_{j=1}^N \left\lbrace z : \pol{n}{\eig{n}{j}} = 0 \right\rbrace,
	\end{equation}
	where $ \eig{n}{j} \in\spec{K_n}$ is the $ j $th eigenvalue of $ K_n $ in non-decreasing order,
	such that $ 0 = \eig{n}{1} < \eig{n}{2} \le \dots \le \eig{n}{N} $,
	and the \emph{characteristic polynomial} associated with $ \eig{}{} $ is
	\begin{equation}\label{eq:char-polynomial}
		\pol{n}{\eig{}{}} \doteq z^{\taun+1} - z^{\taun} + \eig{}{}.
	\end{equation}
\end{lemma}

Matrix $ A_n $ has an eigenvalue at $ 1 $,
corresponding to a root of $ \pol{n}{\eig{n}{1}} $.
For notation convenience,
we denote the largest eigenvalue modulus corresponding to $ \lambda \in \spec{K_n} $ by
\begin{equation}\label{eq:largest-eig-modulus-lambda}
	\rho_n(\lambda) \doteq \max\left\lbrace|z| : \pol{n}{\lambda} = 0 \right\rbrace,
\end{equation}
which allows us to express the convergence rate of~\eqref{eq:controlled-dynamics} as
\begin{equation}\label{eq:conv-rate}
	\rate{n} = \max\left \lbrace\rho_n(\lambda) : \lambda\in\spec{K_n}\setminus\{0\}\right\rbrace.
\end{equation}}

Finally,
to investigate the performance of controllers with different architectures under architecture-dependent delays,
we address maximization of the convergence speed of~\eqref{eq:controlled-dynamics}.

\begin{prob}[Optimal distributed controller]\label{prob:optimal-architecture}
	Given~\cref{ass:architecture,ass:delays,ass:laplacian-K} and an architecture $\net{1}$,
	find the parameter $ {n} $ 
	with optimal feedback gains that optimizes the convergence of~\eqref{eq:controlled-dynamics},
	\begin{equation}\label{eq:prob-optimal-n}
		n^* \in \argmin_n \min_{K_n\revision{\in\mathcal{K}_n}}\;\rate{n},
	\end{equation}
	\revision{where $ \mathcal{K}_n $ collects all matrices that satisfy~\cref{ass:laplacian-K}}.
\end{prob}

\cref{prob:optimal-architecture} requires to evaluate the convergence rate $ \rate{n} $
for $ n \in\{1,\dots,n_\text{max}\} $,
where $\net{n_\text{max}}$ is the complete graph. 
By virtue of the inner minimization in~\eqref{eq:prob-optimal-n} that involves the feedback gains,
we can fairly compare controllers with different architectures.
In~\autoref{sec:design},
we address the optimal design of feedback gains for a given architecture, 
while in~\autoref{sec:simulations} we 
solve~\cref{prob:optimal-architecture} for different 
communication topologies.

\begin{rem}[Impact of delays]
	Under constant
	communication delays,
	the fastest convergence is typically achieved by the densest architecture.
	Here,
	adding communication links yields both the benefit of speeding up information sharing across the network
	(\cref{ass:architecture})
	and the drawback of increasing feedback delays (\cref{ass:delays}).
	As shown later via numerical experiments,
	the latter aspect is crucial role in determining the optimal controller architecture,
	and in particular it can enable fast convergence via sparse controllers.
	Importantly,
	this is different than, \eg~\cite{8485772}
	where carefully removing links improves performance under constant delays:
	indeed,
	we do not remove edges with attention to the performance metric,
	but we optimize the feedback gains of different architectures
	whose topology is not optimized w.r.t. performance.
\end{rem}

\section{Control Design}\label{sec:design}

In this section,
we propose strategies to choose feedback gains in $ K_n $
with a given architecture $ \net{n} $.
Because $ n $ ($ \net{n} $) is fixed,
we omit the script $ n $ in the rest of this section.

\subsection{Stability Analysis}\label{sec:stability-analysis}

The delay-free version of the single-integrator dynamics~\eqref{eq:controlled-dynamics}
is stabilized by any feedback gain matrix $ K $
such that $ I - K $ 
has all eigenvalues inside the unit circle except for the one fixed at $ 1 $.
The same stability requirement holds
if the delay-free inertia term $ \x{}{k} $ in the right-hand side  of~\eqref{eq:controlled-dynamics} is also delayed
so that the dynamics are simply ``stretched" in time by a factor $ \tau $,
as assumed in~\cite{gupta2010delay}.
Conversely,
more complex stability conditions hold for system~\eqref{eq:controlled-dynamics},
which were derived in~\cite{ballotta2023tcns} and are reported below for convenience.\footnote{
	Even though reference~\cite{ballotta2023tcns} is concerned with mean-square stability of 
	the stochastically forced version of~\eqref{eq:controlled-dynamics} with additive noise,
	stability conditions are the same because they depend only on the system modes (eigenvalues).
}

\begin{prop}[\!\!{\cite[Proposition~4]{ballotta2023tcns}}]
	System~\eqref{eq:controlled-dynamics} is stable if and only if,
	for $ j = 2,\dots,N $,
	\begin{equation}\label{eq:stability-condition-eig}
		0 < \gpos_j < \up{\tau},
		\qquad \up{\tau} \doteq 2\sin\left(\dfrac{\pi}{2}\dfrac{1}{2\tau+1}\right).
	\end{equation}
\end{prop}

In words,
condition~\eqref{eq:stability-condition-eig} shows that the delay in~\eqref{eq:controlled-dynamics}
triggers unstable modes if control actions are too aggressive.
In particular,
larger delays force more conservative control,
as seen by the bound $ \up{\tau} $ which is decreasing with $ \tau $.
Notably,
the bound $ \up{\tau} $ does not depend on the network topology.

Straightforward application of Gershgorin theorem yields the following result,
which is amenable to distributed implementation
and restrict the usual conditions on edge weights.


\begin{cor}[Stabilizing uniform gains]\label{cor:stability-conditions-gains-uniform}
	Let $ \revision{K = -gL} $
	with $ L $ the Laplacian matrix of $ \net{} $.
	System~\eqref{eq:controlled-dynamics} is stable if
	\begin{equation}\label{eq:stability-condition-gains-uniform}
		0 < \revision{g} < \dfrac{\up{\tau}}{2\dmax{}}, \qquad \dmax{} \doteq \max_i \deg{}{i}.
	\end{equation}
\end{cor}


\subsection{Optimization of Feedback Gains}\label{sec:optimization}

\cref{cor:stability-conditions-gains-uniform}
give conditions for simple choices of feedback gains that can be implemented locally at nodes.
For example,
the standard choice $ \revision{g} = \frac{1}{2\dmax{}+1} $ becomes $ \revision{g} = \frac{\up{\tau}}{2\dmax{}+1} $.

We next address the optimization of feedback gains in~\cref{prob:optimal-architecture}.
This corresponds to the inner minimization of $ K_n $ in~\eqref{eq:prob-optimal-n} for a given architecture $ \net{n} $
(\ie for given $n$).
We focus on a centralized design that requires knowledge of the eigenvalues of $ K $,
deferring a distributed design to future work.

\begin{prob}[Optimal structured controller]\label{prob:gains}
	Find a feedback gain matrix $ K $ that minimizes the convergence rate of~\eqref{eq:controlled-dynamics},
	\begin{equation}\label{eq:optimization-gains}
		K^* \in \argmin_{K\in\mathcal{K}} \rate{}. 
	\end{equation}
\end{prob}

It is known that~\cref{prob:gains} is nonconvex
\revision{and nonsmooth~\cite{bagirov2014introduction}}.
Nonetheless,
\revision{the specific structure of our problem allows us to retrieve the solution of~\eqref{eq:optimization-gains},
as described next.}

A related problem is faced in~\cite{carli11tac} for delay-free double integrators,
where the authors note that
$ \rho(\lambda) $
is decreasing for $ \lambda < \lambda_{\text{th}} $
and increasing for $ \lambda > \lambda_{\text{th}} $,
for a threshold $ \lambda_{\text{th}} $.
This immediately yields that the SLEM of $ A $
is a root of either $ \pol{}{\lambda_2} $ or $ \pol{}{\lambda_N} $,
which simplifies the optimization problem into a cascade composed of a convex SDP
followed by an algebraic equation.

\begin{figure}
	\centering
	\begin{minipage}[l]{.5\linewidth}
		\centering
		\includegraphics[height=.85\linewidth]{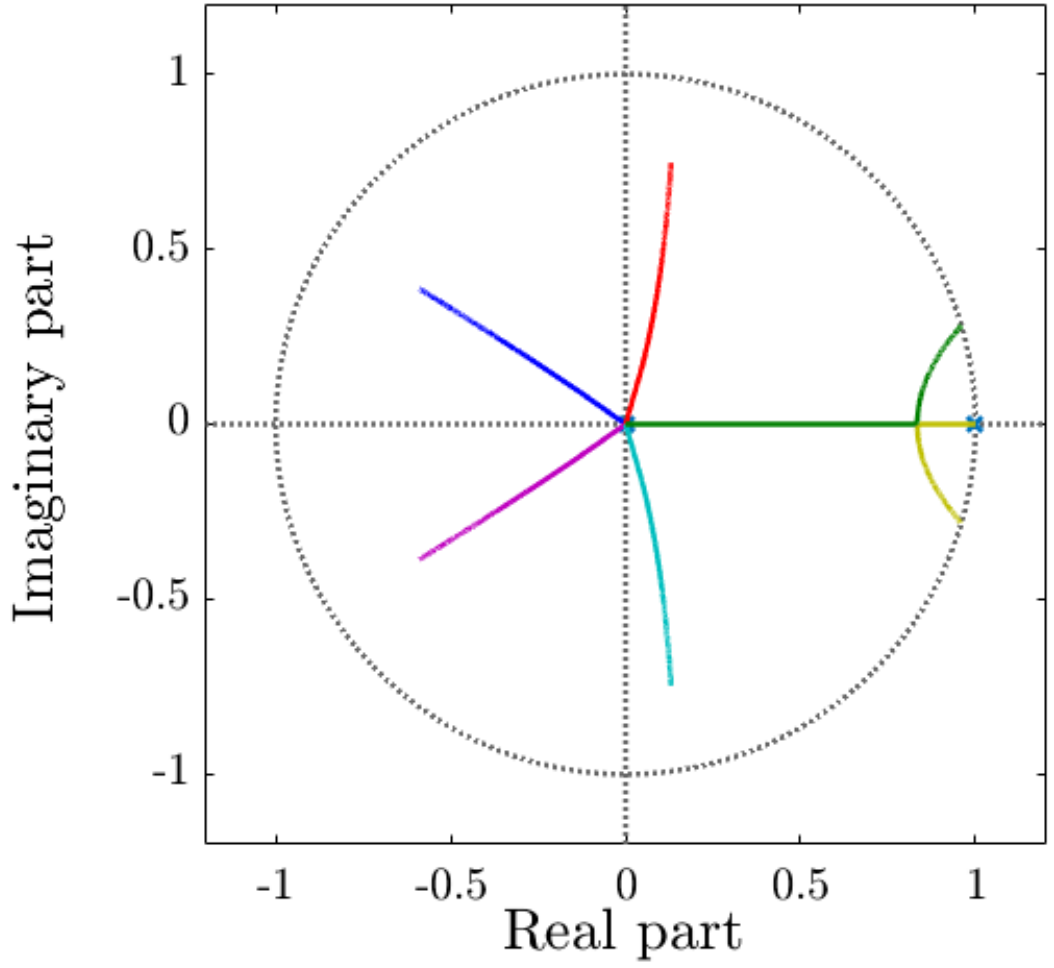}
		\label{fig:rl_tau_5}
	\end{minipage}%
	\hfil
	\begin{minipage}[r]{.5\linewidth}
		\centering
		\includegraphics[height=.85\linewidth]{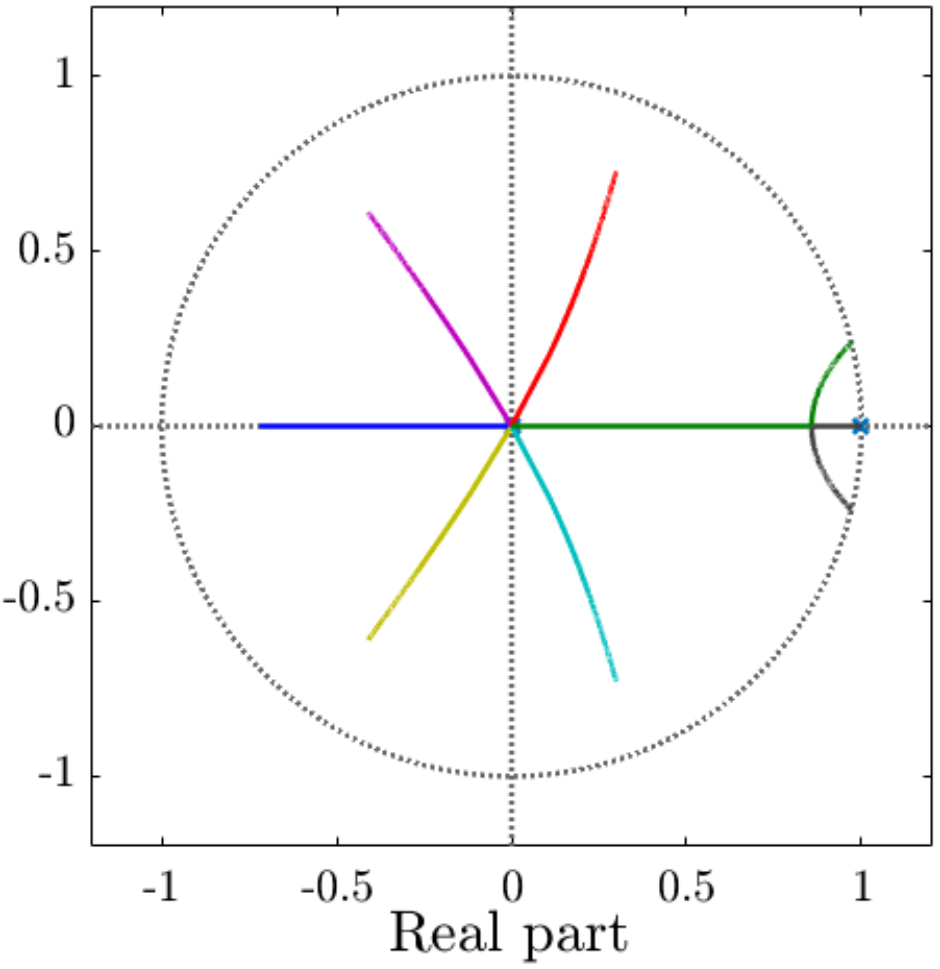}
		\label{fig:rl_tau_6}
	\end{minipage}
	\caption{Root locus of $ \pol{}{\lambda} $ with $ \lambda \in [0,\up{\tau}] $
		for $ \tau = 5 $ (left) and $ \tau = 6 $ (right).
		All branches expand from $ z = 0 $
		except for the one starting at $ z = 1 $.
	}
	\label{fig:rl}
\end{figure}

The key move that allowed the authors in~\cite{carli11tac}
to assess monotonicity properties of $\rho(\eig{}{})$
was its explicit calculation,
which in this case is not possible for arbitrary $ \tau $.
However,
the same behavior proved in~\cite{carli11tac} can be observed here
if we consider the root locus associated with the characteristic polynomial $ \pol{}{\lambda} $,
where $ \lambda $ acts as feedback gain.
Two typical root loci are shown in~\autoref{fig:rl},
one for odd $ \tau $ and one for even $ \tau $.
When $ \lambda = 0 $,
there is one root at $ z = 0 $ with multiplicity $ \tau $ and one simple root at $ z = 1 $.
As $ \lambda $ increases,
the latter decreases along the real axis
while all other roots grow in modulus.
In particular,
there is a real positive root that 
increases along the real axis till it meets the other positive real root,
after which the two corresponding branches enter the complex plane
and expand towards the unit circle. 
This means that the largest modulus of roots of the characteristic polynomial $ \pol{}{\lambda} $
is first decreasing
(when the largest real solution decreases monotonically from $ 1 $ along the real axis)
and then increasing
(corresponding to either a pair of complex branches or to the negative real root,
if present,
as soon as the corresponding modulus becomes larger than the largest positive real root).

The next lemmas formalize the discussion above.

\begin{lemma}[Monotonicity of real eigenvalues]\label{prop:monotonicity-real-roots}
	Consider the following definition
	associated with real roots of $ \pol{}{\lambda} $,
	\begin{equation}\label{eq:rate-real-roots}
		\rho_\Re(\lambda) \doteq \max\left\lbrace|z| : \pol{}{\lambda} = 0, z \in\Real{}\right\rbrace.
	\end{equation}
	If $ \rho_\Re(\lambda) \neq\emptyset $,
	there exists $ \lambda_{\text{th},\Re} $ such that 
	$ \rho_\Re(\lambda) $ is decreasing for $ \lambda < \lambda_{\text{th},\Re} $
	and increasing for $ \lambda > \lambda_{\text{th},\Re} $.
\end{lemma}
\begin{proof}
	See~\cref{app:monotonicity-real-eigenvalues}.
\end{proof}

\begin{lemma}[Monotonicity of complex eigenvalues]\label{lem:monotonicity-complex-roots}
	All complex roots of $ \pol{}{\lambda} $
	with nonzero imaginary part
	have modulus increasing with $ \lambda $.
\end{lemma}
\begin{proof}
	See~\cref{app:monotonicity-complex-eigenvalues}.
\end{proof}

\begin{figure}
	\centering
	\includegraphics[height=.4\linewidth]{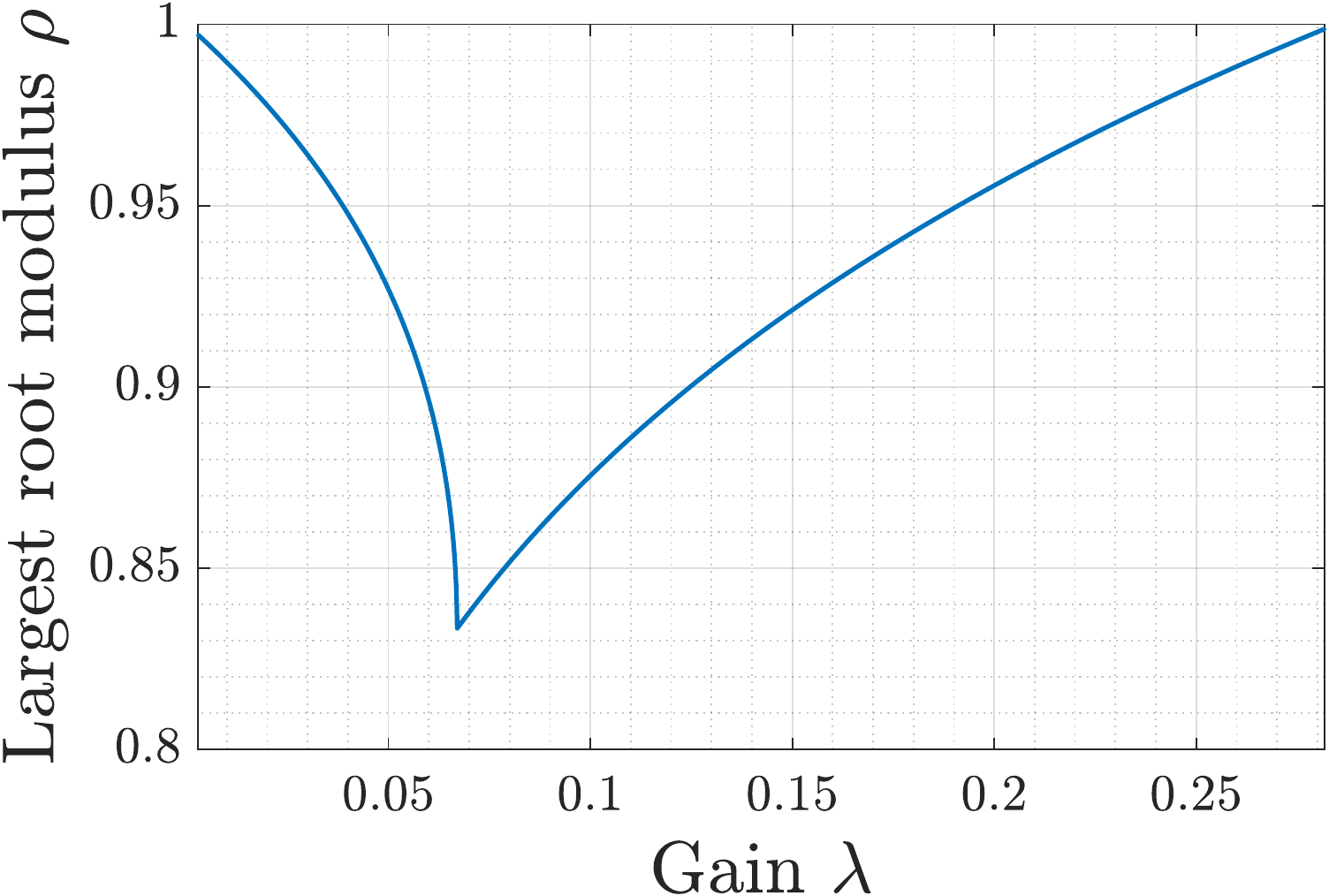}
	\caption{Graphic of $ \rho(\lambda) $ with $ \lambda \! \in \!(0,\up{\tau}) $ for $ \tau = 5 $.}
	\label{fig:rho}
\end{figure}


\cref{prop:monotonicity-real-roots,lem:monotonicity-complex-roots} straightly lead to the following proposition,
which is key to our \revision{solution approach} (see~\autoref{fig:rho}).

\begin{prop}[{Convergence rate}]\label{prop:slem}
	There exists $ \lambda_{\text{th}} $ such that $ \rho(\lambda) $
	is decreasing for $ 0 < \lambda < \lambda_{\text{th}} $
	and increasing for $ \lambda_{\text{th}} < \lambda < \up{\tau} $.
	\revision{Hence,
	it holds $ \rate{} = \max\left\lbrace\rho(\lambda_2), \rho(\lambda_N)\right\rbrace $.}
\end{prop}

By virtue of~\cref{prop:slem},
\cref{prob:gains} can be \revision{solved} \revision{via the same technique proposed in~\cite{carli11tac}},
which is summarized next.

Consider the subset of $ \mathcal{K} $ comprising matrices $ K $
with eigenvalue $ \lambda_2(K) = 1 $,
denoted by $ \bar{\mathcal{K}} $.
\revision{Then,
it holds
$ \mathcal{K} = \cup_{\beta\ge0}\beta\bar{\mathcal{K}} $.
In particular,
any $ K \in \mathcal{K} $ can be written as $K = \beta\bar{K}$ for some suitable $ \bar{K}\in\bar{\mathcal{K}} $ and $ \beta > 0 $.}
By virtue of this observation and~\cref{prop:slem},
\revision{and denoting the $i$th eigenvalue of $\bar{K}$ by $\bareig{}{i}$,}
problem~\eqref{eq:optimization-gains} can be equivalently written as
\begin{equation}\label{eq:optimization-gains-decoupled}
	\left\lbrace \bar{K}^*,\beta^*\right\rbrace \in 
	\argmin_{\bar{K}\in\bar{\mathcal{K}},\,\beta>0}
	\max\left\lbrace\rho(\beta),\rho(\beta\bareig{}{N})\right\rbrace,
\end{equation}
the solution to~\eqref{eq:optimization-gains} being retrieved as $ K^* = \beta^*\bar{K}^* $.

It can be seen that~\eqref{eq:optimization-gains-decoupled} \revision{can} in fact \revision{be} \emph{decoupled}
with respect to the two variables $ \bar{K} $ and $ \beta $,
and it can be solved by first optimizing over $ \bar{K} $ 
\revision{(with $\beta=1$)}
and then over $ \beta $
\revision{(given the solution $\bar{K}^*$ of the first sub-problem)}.
This is because $ \rho(\beta) $ is independent of $ \bar{K} $
and $ \rho(\beta\bareig{}{N}) $ is parametric in $ \bar{K} $.
The set $ \bar{\mathcal{K}} $ is nonconvex,
however,
the optimization for $ \bar{K} $ \revision{can be exactly recast}
into the following convex \revision{SDP},
noting that the solution features $ \bareig{*}{2} = 1 $
(with $\bareig{*}{i}$ the $i$th eigenvalue of $\bar{K}^*$),
\begin{equation}\label{eq:optimization-gains-decoupled-K}
	\bar{K}^*\in \argmin_{\bar{K}\in\mathcal{K},\,\bareig{}{2} \ge 1} \bareig{}{N}.
\end{equation}
Finally,
the optimal scaling parameter $\beta^*$ is retrieved as the solution to the following optimization problem,
\begin{argmini!}<b>
	{\substack{\beta>0}}
	{\max\left\lbrace\rho(\beta),\rho(\beta\bareig{*}{N})\right\rbrace\protect\label{eq:optimization-gains-decoupled-beta-objective}}
	{\label{eq:optimization-gains-decoupled-beta}}
	{\beta^*=}
	\addConstraint{\beta\bareig{*}{N}}{< \up{\tau}.\protect\label{eq:optimization-gains-decoupled-beta-constraint}}
\end{argmini!}
The linear constraint~\eqref{eq:optimization-gains-decoupled-beta-constraint} ensures that 
the stability condition~\eqref{eq:stability-condition-eig} is satisfied.
By virtue of~\cref{prop:slem},
\revision{minimizing the cost in}~\eqref{eq:optimization-gains-decoupled-beta-objective}
amounts to finding the unique solution of
\begin{equation}\label{eq:equation-beta}
	\rho(\beta) = \rho(\beta\bareig{*}{N}),
\end{equation}
\revision{which can be solved by the bisection method.
If the solution of~\eqref{eq:equation-beta} complies with~\eqref{eq:optimization-gains-decoupled-beta-constraint},
then~\eqref{eq:optimization-gains-decoupled-beta} is solved.}
If this is not the case,
\revision{\cref{prop:slem} implies that both the second smallest eigenvalue and the largest eigenvalue of $K$
need to be smaller than $\eig{}{\text{th}}$,
which means that $R = \rho(\beta)$}
and thus $ \beta $ shall be chosen as close as possible 
to the upper bound $ \nicefrac{\up{\tau}}{\bareig{*}{N}} $. 

\begin{rem}[Optimal uniform gains]\label{rem:optimal-uniform-gains}
	\revision{Let} the feedback gain matrix be chosen as $ \revision{K = -gL} $,
	 \revision{then the uniform} gain $ \revision{g} $ 
	can be optimized akin $ \beta $ according to~\eqref{eq:optimization-gains-decoupled-beta}
	\revision{by letting} $ \bar{K}^{*} = L $.
\end{rem}

\section{Numerical Experiments}\label{sec:simulations}

In this section,
we numerically solve~\cref{prob:optimal-architecture} \revision{for different networks}
\revision{with the goal of investigating the optimal controller architecture for a given initial structure $ \net{1} $.}
Because we now consider multiple architectures,
we use the subscript $ n $ to specify the architecture $ \net{n} $ \revision{constructed from $\net{1}$ according to~\cref{ass:architecture}
(recall that a larger $n$ corresponds to a larger number of links)}.
\revision{In view of~\eqref{eq:prob-optimal-n},
we aim to find an optimal parameter $n^*$ 
such that 
the controller architecture $\net{n^*}$
yields the minimum convergence rate $\rate{n^*}$ of the dynamics~\eqref{eq:controlled-dynamics}.}

All networks have $ N = 100 $ agents.
\revision{The parameter $ n $ 
\revision{(referred to as \textit{number of hops} in the $x$-axis of Figs.~\ref{fig:conv_rate_reg_3_N_100_tau_n}--\ref{fig:conv_rate_rand_N_100_tau_n})}
ranges within $\{1,\dots,n_{\text{max}}\}$
such that $ \net{n_\text{max} + 1} $ is fully connected.}\footnote{
	When mentioning ``number of hops" we actually refer to new communication links
	added to the initial architecture
	according to~\cref{ass:architecture}.
}

In each scenario,
we consider the following design strategies for the feedback gains:
uniform gain $ \revision{g} = \frac{\up{\tau}}{2\dmax{n}+1} $ (\textit{not optimized});
uniform gain $ \revision{g}^* $ computed according to~\cref{rem:optimal-uniform-gains} (\textit{optimized});
multiple gains addressing~\cref{prob:gains} and computed according to~\eqref{eq:optimization-gains-decoupled-K}--\eqref{eq:optimization-gains-decoupled-beta}.
\revision{In all cases,
the solution of the equation~\eqref{eq:equation-beta} actually meets the stability condition~\eqref{eq:optimization-gains-decoupled-beta-constraint},
which makes the optimization problem~\eqref{eq:optimization-gains-decoupled-beta} over $\beta$ feasible.}

Figure~\ref{fig:conv_rate_reg_3_N_100_tau_n} shows the convergence rate $ \rate{n} $
with the initial network $ \net{1} $ chosen as a $ 3 $-regular graph,
with communication delays increasing linearly with $n$
(which is inspired by multi-hop communication).
We first note that all curves are consistent with our optimization of feedback gains:
the non-optimized uniform gains induce the worst (largest) convergence rates,
whereas progressive improvement is observed passing to optimized uniform gains
and eventually to multiple optimized gains.
One can see that the smallest convergence rate is always achieved by distributed
(not fully connected) architectures.
Interestingly,
the optimal \revision{parameter $n^*$} is quite different according to the chosen design strategy for the feedback gains:
in particular,
with both uniform not optimized gains and multiple gains
the fastest convergence is achieved for $ n^* = 5 $,
whereas choosing a single optimized gain yields $ n^* = 7 $ as the optimal controller architecture.
This calls for carefulness when performing control design
and in particular when \textit{jointly} designing both feedback gains and controller architecture.
Also,
note that some suboptimal choices of gains actually induce better or comparable performance if complemented with
suitably connected architecture:
for example,
architecture $ \net{4} $ with non-optimized uniform gains performs slightly better than
sparser architecture $ \net{3} $ with optimized uniform gains,
whereas architecture $ \net{7} $ with optimized uniform gains
is better than denser architecture $ \net{8} $ with optimized multiple gains.

Figure~\ref{fig:conv_rate_reg_3_N_100_tau_n2} shows performances when delays increase quadratically
with $ n $,
which is inspired by the setup in~\cite{gupta2010delay} where
number of links and delays increase with the transmission power.
It can be seen that all curves are ``pulled" toward bottom-left and that
the optimal architectures are sparse,
featuring $ n^* = 1 $ for non-optimized uniform gains and $ n^* = 2 $ for optimized gains.
This is because,
compared to~\autoref{fig:conv_rate_reg_3_N_100_tau_n},
delays grow faster and degrade the dynamics more quickly with $ n $,
forcing the controller to act in a more conservative way.

Figure~\ref{fig:conv_rate_reg_4_N_100_tau_n} shows convergence rates with denser architectures,
where the initial network $ \net{1} $ is a 4-regular graph.
The qualitative behavior is similar to~\autoref{fig:conv_rate_reg_3_N_100_tau_n}.
The two curves corresponding to non-optimized uniform gains
and multiple gains exhibit nontrivial points of minimum respectively at $ n^* = 2 $ and $ n^* = 3 $,
whereas the optimal architecture for uniform optimized gains is almost fully connected.
Analogous considerations hold for~\autoref{fig:conv_rate_reg_4_N_100_tau_n2},
with trends similar to~\autoref{fig:conv_rate_reg_3_N_100_tau_n2}.

Finally,
tests performed on a random graph in~\autoref{fig:conv_rate_rand_N_100_tau_n}
(with sparse and dense areas)
with linear delay increase
also show behavior similar to~\cref{fig:conv_rate_reg_3_N_100_tau_n,fig:conv_rate_reg_4_N_100_tau_n},
suggesting a consistent pattern that mostly depend on the delay rate $ f(\cdot) $.

\begin{figure}
	\centering
	\includegraphics[height=.5\linewidth]{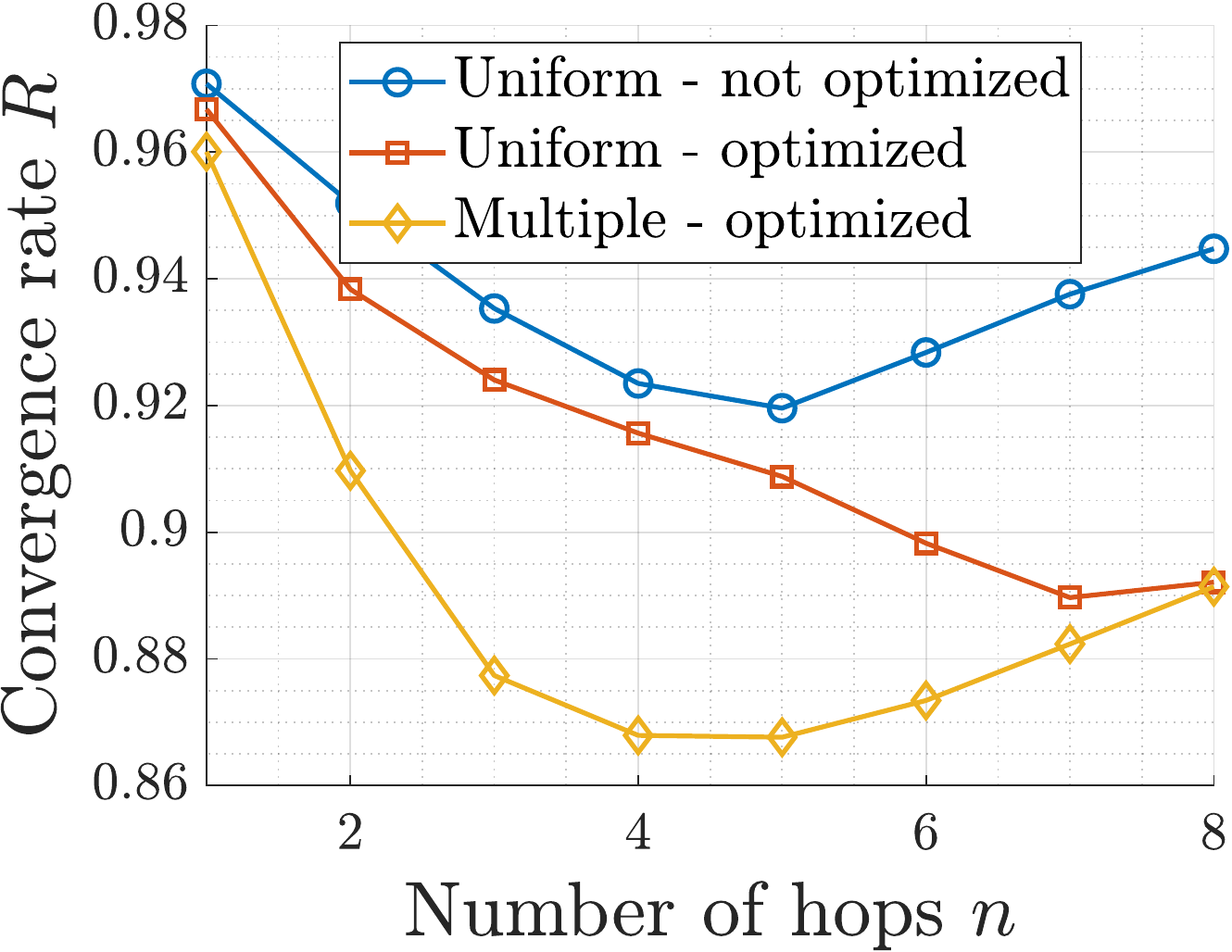}
	\caption{Convergence rate with $ 3 $-regular graph $ \net{1} $, $ N = 100 $, $ \taun = n $.}
	\label{fig:conv_rate_reg_3_N_100_tau_n}
\end{figure}

\begin{figure}
	\centering
	\includegraphics[height=.5\linewidth]{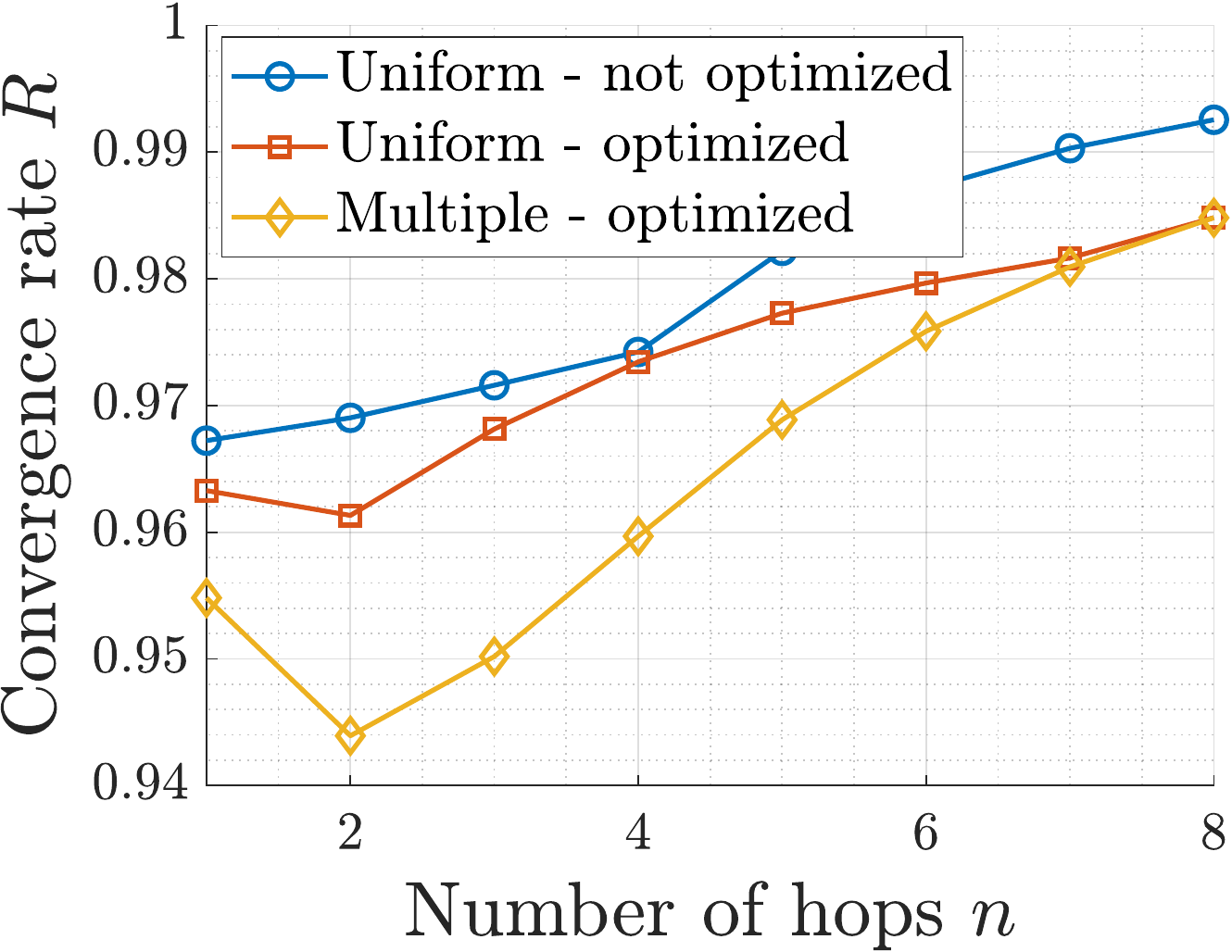}
	\caption{Convergence rate with $ 3 $-regular graph $ \net{1} $, $ N = 100 $, $ \taun = n^2 $.}
	\label{fig:conv_rate_reg_3_N_100_tau_n2}
\end{figure}

\begin{figure}
	\centering
	\includegraphics[height=.5\linewidth]{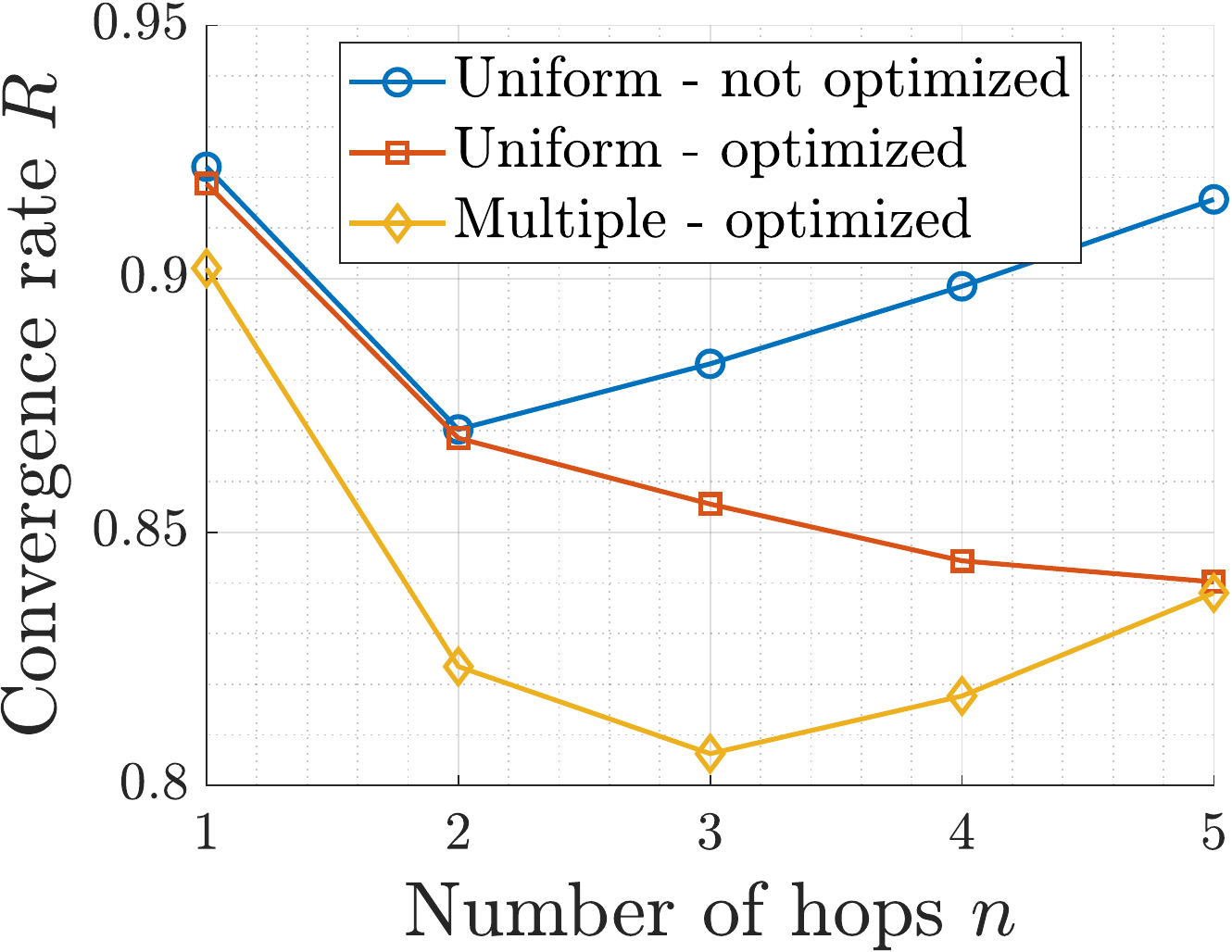}
	\caption{Convergence rate with $ 4 $-regular graph $ \net{1} $, $ N = 100 $, $ \taun = n $.}
	\label{fig:conv_rate_reg_4_N_100_tau_n}
\end{figure}

\begin{figure}
	\centering
	\includegraphics[height=.5\linewidth]{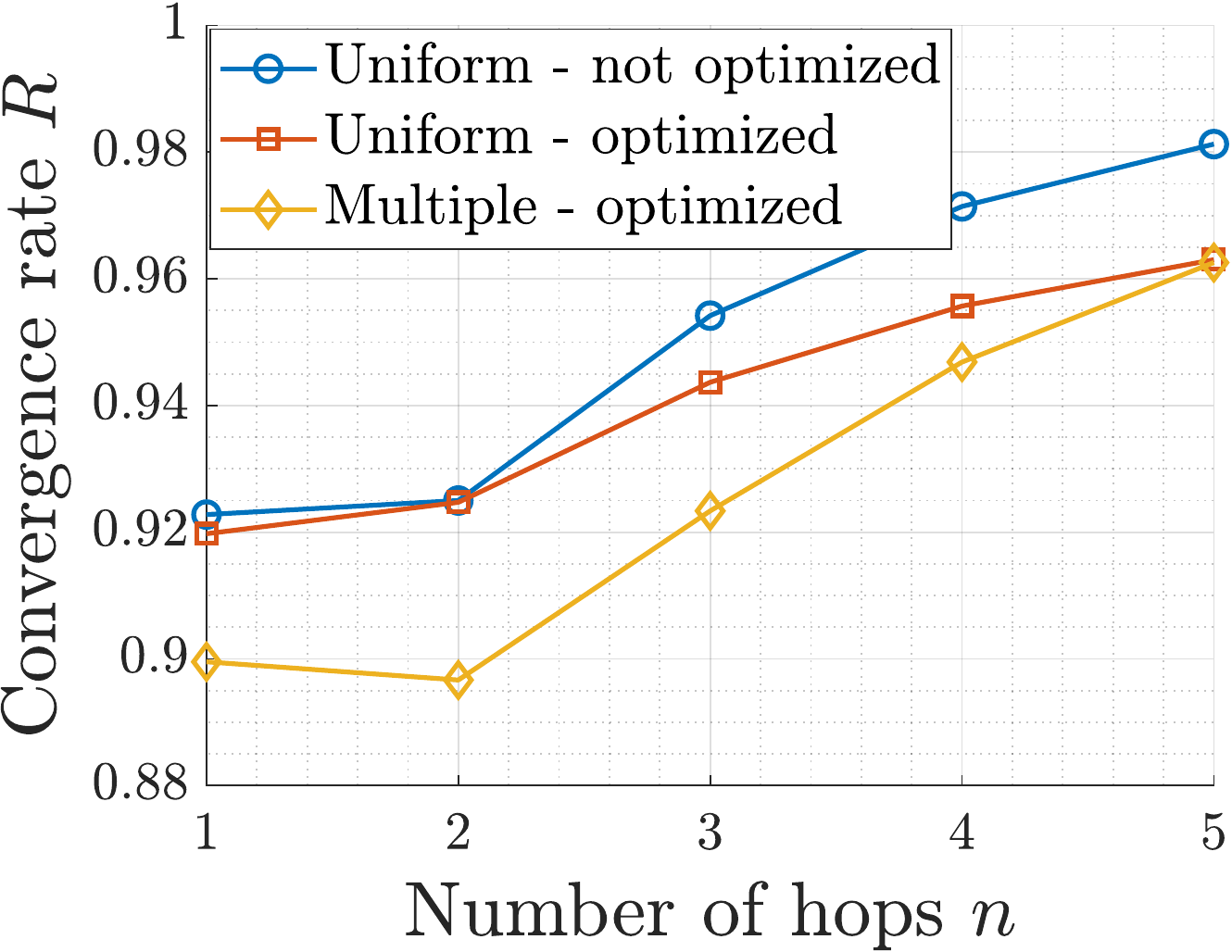}
	\caption{Convergence rate with $ 4 $-regular graph $ \net{1} $, $ N = 100 $, $ \taun = n^2 $.}
	\label{fig:conv_rate_reg_4_N_100_tau_n2}
\end{figure}

\begin{figure}
	\centering
	\begin{minipage}[l]{.5\linewidth}
		\centering
		\includegraphics[height=.75\linewidth]{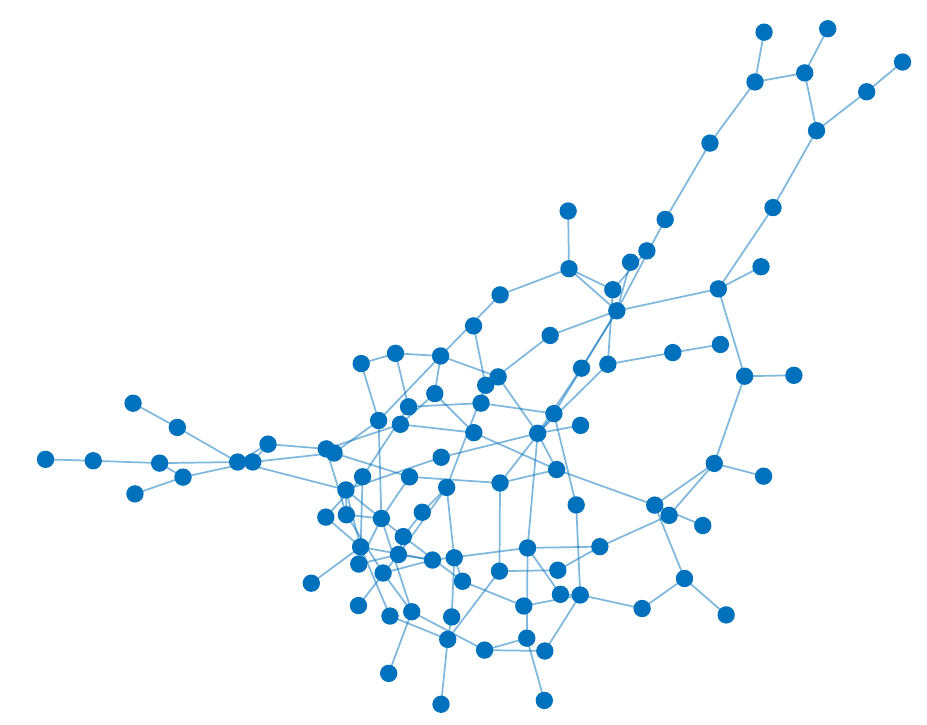}
	\end{minipage}%
	\hfil
	\begin{minipage}[r]{.5\linewidth}
		\centering
		\includegraphics[height=.75\linewidth]{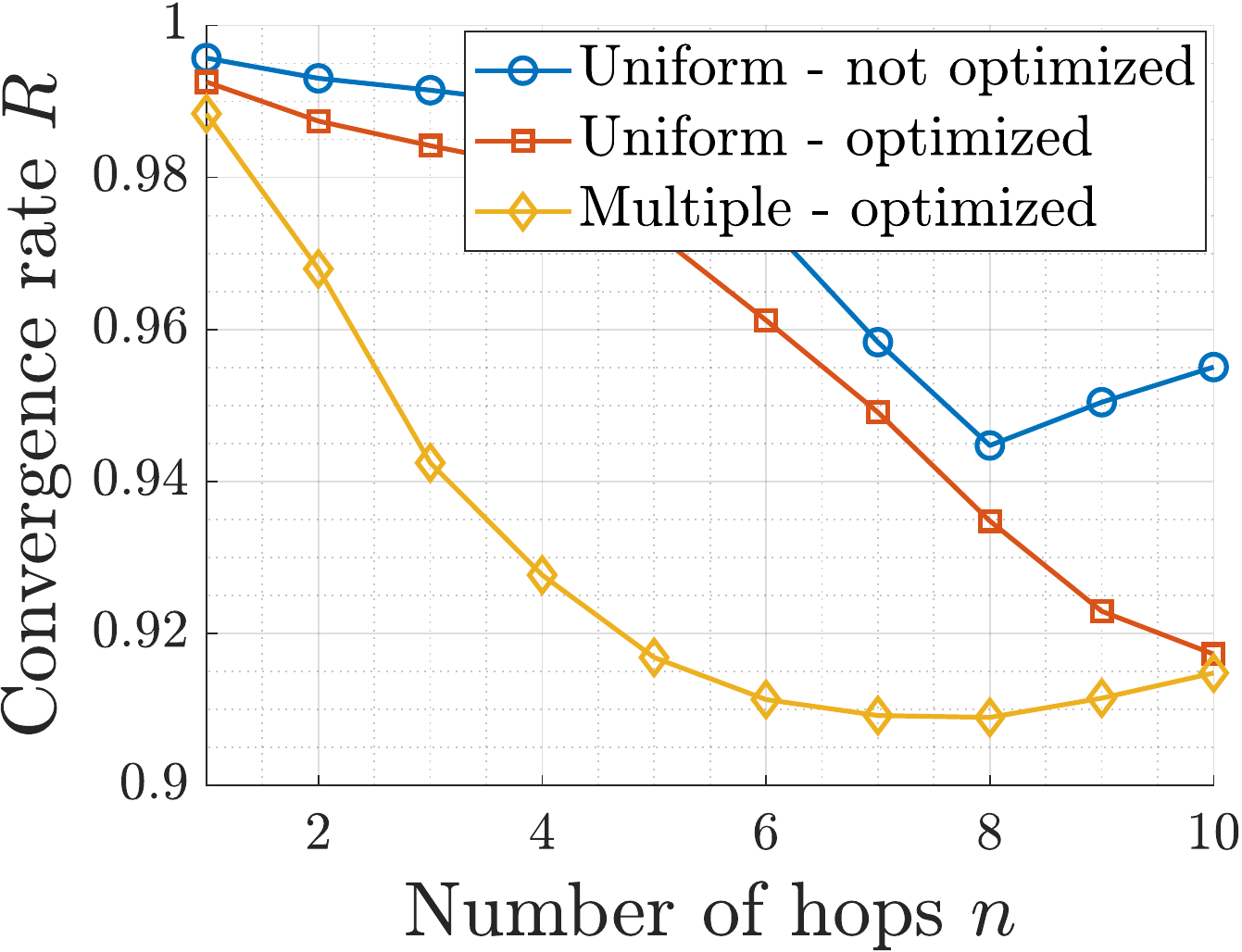}
	\end{minipage}
	\caption{Convergence rate with random graph $ \net{1} $, $ N = 100 $, $ \taun = n $.}
	\label{fig:conv_rate_rand_N_100_tau_n}
\end{figure}

\section{Conclusion}\label{sec:conclusion}

In this \paperType,
we have shown
that optimal distributed controllers
with respect to the convergence rate of a consensus protocol
need not be fully connected.
In fact,
the fastest convergence for a network of single integrators 
is achieved via sparse architectures
if feedback delays increase fast enough with the number of communication \revision{links}.
	
	\bibliographystyle{IEEEtran}

	\appendix
	\numberwithin{equation}{subsection}

\subsection{Proof of~\autoref{prop:monotonicity-real-roots}}\label{app:monotonicity-real-eigenvalues}

First,
we show that the negative root,
when it exists,
is decreasing.
Second,
we show the monotonic behavior of the two positive roots of $ \pol{}{\lambda} $,
as long as they exist.
The combination of those facts yields the claim in~\cref{prop:monotonicity-real-roots}.

We use the implicit function theorem to compute derivative of roots.
The hypotheses of the theorem are satisfied as long as $ z\neq0 $
and the two positive roots are not coincident,
which is enough to prove the claim having those two sets zero Lebesgue measure.
The derivative of a root $ z\in\Real{} $ is
\begin{equation}\label{eq:derivative-real-roots}
	\dfrac{\mathrm{d}z}{\mathrm{d}\lambda}(\lambda) = \dfrac{\partial_\lambda\pol{}{\lambda}}{\partial_z\pol{}{\lambda}}
													= \dfrac{-1}{(\tau+1)z^\tau - \tau z^{\tau-1}}.
\end{equation}

\myParagraph{Negative root}
Polynomial $ \pol{}{\lambda} $ has a negative root for $ \lambda > 0 $ and even $ \tau $.
This can be seen,
for example,
via Descartes' rule of signs
and the fact that $ h(-z;\lambda) $ has one real root
being its coefficients real.
Then,
derivative~\eqref{eq:derivative-real-roots} is always negative for $ z < 0 $.
Conversely,
again Descartes' rule of signs proves that no real negative root exist for odd $ \tau $.

\myParagraph{Positive roots}
When $ \lambda = 0 $,
the simple root at $ z = 1 $
has negative derivative. 
On the other hand,
for $ \lambda \gtrsim 0 $,
the real root at $ z \gtrsim 0 $ is an implicit function of $ \lambda $ with positive derivative.
The denominator $ \partial_z\pol{}{\lambda} $ of~\eqref{eq:derivative-real-roots} is continuous in $ z $ and thus in $ \lambda $
(over appropriate domain):
hence,
the larger root is decreasing with $ \lambda $ as long as $ \partial_z\pol{}{\lambda} > 0 $
and
the smaller root is increasing with $ \lambda $ as long as $ \partial_z\pol{}{\lambda} < 0 $.
The denominator becomes zero for $ \lambda = \lambda_{\text{b}} \doteq \frac{\tau^\tau}{(\tau+1)^{\tau+1}} $ 
when $ z = \frac{\tau}{\tau + 1} $,
and the derivatives switch sign for $ \lambda > \lambda_{\text{b}} $.
However,
this event is not possible.
Indeed,
$ \pol{}{\lambda} $ admits real positive roots only within the interval $ [0,\frac{\tau}{\tau + 1}] $.
This is because $ \frac{\tau}{\tau + 1} $ is the point of minimum of $ \pol{}{\lambda} $,
which is zero for $ \lambda = \lambda_{b} $ and positive for $ \lambda > \lambda_{b} $.
We conclude that $ \pol{}{\lambda} $ has two real positive roots for $ 0 < \lambda \le \lambda_{b} $,
one monotonically increasing from $ \epsilon $ to $ \frac{\tau}{\tau + 1} $,
for any $ \epsilon > 0 $,
and one monotonically decreasing from $ 1 $ to $ \frac{\tau}{\tau + 1} $.

\autoref{prop:monotonicity-real-roots} holds true with $ \lambda_{\text{th},\Re} $
equal to (i) $ \lambda_{\text{b}} $
(with $\rho_\Re$($\lambda$) monotonically increasing), 
for odd $ \tau $,
and (ii) the minimum between $ \lambda_{\text{b}} $ and the $ \lambda $
such that the negative root is larger in modulus than the positive roots (if any),
for even $ \tau $.

\subsection{Proof of~\autoref{lem:monotonicity-complex-roots}}\label{app:monotonicity-complex-eigenvalues}

We study the sign of the derivative of each complex root $ z\in\mathbb{C} $ of $ \pol{}{\lambda} $. 
We use again the implicit function theorem and rewrite derivative~\eqref{eq:derivative-real-roots} for $ \rho > 0 $,
where $ z = \rho\e^{j\theta} $,
\begin{equation}\label{eq:derivative-complex-roots}
	\dfrac{\mathrm{d}z}{\mathrm{d}\lambda}(\lambda)
	= -\dfrac{(\tau+1)\rho^\tau\e^{-j\tau\theta} - \tau \rho^{\tau-1}\e^{-j(\tau-1)\theta}}{|(\tau+1)z^\tau - \tau z^{\tau-1}|^2}.
\end{equation}
Standard derivation rules and algebraic manipulations yield
\begin{equation}\label{eq:derivative-complex-roots-modulus}
	\begin{aligned}
		\dfrac{1}{2}\dfrac{\mathrm{d}\rho^2}{\mathrm{d}\lambda}(\lambda) 
		&= \Re(z)\dfrac{\mathrm{d}\Re(z)}{\mathrm{d}\lambda}(\lambda) + \Im(z)\dfrac{\mathrm{d}\Im(z)}{\mathrm{d}\lambda}(\lambda)\\
		&= -\dfrac{(\tau+1)\rho^{\tau+1}\cos((\tau+1)\theta) - \tau \rho^{\tau}\cos(\tau\theta)}{|(\tau+1)z^\tau - \tau z^{\tau-1}|^2},
	\end{aligned}
\end{equation}
from which it follows that $ \frac{\mathrm{d}\rho}{\mathrm{d}\lambda}(\lambda) > 0 $ is equivalent to
\begin{equation}\label{eq:derivative-complex-roots-modulus-positive}
	(\tau+1)\rho\cos((\tau+1)\theta) - \tau \cos(\tau\theta) < 0.
\end{equation}
By considering the equality $ \Im(\pol{}{\lambda}) = 0 $
and assuming that $ \sin((\tau+1)\theta) \neq 0 $,
we get the following relationship,
\begin{equation}\label{eq:relationship-modulus-phase}
	\rho = \dfrac{\sin(\tau\theta)}{\sin((\tau+1)\theta)}.
\end{equation}
Note that the set of roots such that $ \sin((\tau+1)\theta) = 0 $ is discrete with zero Lebesgue measure
and does not impact monotonicity of $ \rho $.
Then,
combining~\eqref{eq:derivative-complex-roots-modulus-positive} and~\eqref{eq:relationship-modulus-phase}
yields
\begin{equation}\label{eq:derivative-complex-roots-modulus-positive-1}
	(\tau+1)\dfrac{\sin(\tau\theta)}{\sin((\tau+1)\theta)}\cos((\tau+1)\theta) - \tau \cos(\tau\theta) < 0,
\end{equation}
which standard manipulations transform to
\begin{equation}\label{eq:derivative-complex-roots-modulus-positive-2}
	\sin(\tau\theta)\cos(\tau\theta)\cos\theta - \sin^2(\tau\theta)\sin\theta < \tau,
\end{equation}
which is always true for $ \tau \ge 2 $.
The case $ \tau = 1 $ is trivially verified by explicitly computing the roots of $ \pol{}{\lambda} $.
	
\end{document}